\documentclass[reqno]{amsart} \usepackage{amscd}
\usepackage{epsf}
\newtheorem{theorem}{Theorem}[section]
\newtheorem{proposition}[theorem]{Proposition}

\newtheorem{corollary}[theorem]{Corollary}
\theoremstyle{remark} 

\begin{document}

\title[Solitons from Minimal surfaces]{One parameter family of solitons from minimal surfaces }

\author{Rukmini Dey and Pradip Kumar}
\address{School of Mathematics\\
Harish Chandra Research Institute\\
Allahabad, 211019, India\\
rkmn@mri.ernet.in, pmishra@mri.ernet.in}

\begin{abstract}
In this paper, we discuss a one parameter family of complex Born-Infeld 
solitons arising
from a one parameter family of minimal surfaces. 
 The process enables us to generate a new solution of the B-I equation
 from a given complex solution of a special type (which are abundant).   
We illustrate this with many examples.
We find that the action or the energy  of this family of solitons 
remains invariant in this family and find that the well-known Lorentz symmetry
of the B-I equations is responsible for it. 
\end{abstract}

\maketitle

\section{Introduction}

In a previous paper ~\cite{D}, using hodographic coordinates, we found
the general solution to the minimal surface equation, namely a variant of the 
Weirstrass-Enneper representation of the minimal surface.
This was done by wick rotating the general Born-Infeld soliton solution by 
Barbishov and Chernikov 
discussed in the last section of ~\cite{W}. Underlying this, there was the observation that  the minimal surface equation
$$(1+ \phi_t^2) \phi_{xx} - 2 \phi_x \phi_t \phi_{xt} + (1 + \phi_x^2) \phi_{tt} =0$$  
and the Born-Infeld equation 
$$(1- \phi_t^2) \phi_{xx} + 2 \phi_x \phi_t \phi_{xt} - (1 + \phi_x^2) \phi_{tt}
=0$$
can be obtained one from the other by wick rotation of the variable $t$.

We know that if $X (r, \bar{r})= (x_1(r, \bar{r}) , t_1(r, \bar{r}), \phi_1(r, \bar{r})) $ 

and $Y(r, \bar{r}) = (x_2(r, \bar{r}), t_2(r, \bar{r}), \phi_2(r, \bar{r})) $ are two minimal surfaces in isothermal coordinates $(r_1, r_2)$, where $r= r_1 + i r_2,$  which are  
harmonic conjugate  to 
each other, then $ \tilde{X} (r, \bar{r}, \theta)  = cos(\theta) X(r, \bar{r})  + sin (\theta) Y (r, \bar{r}) $ is again a 
minimal surface 
for each $\theta$, ~\cite{DoC}, page 213. Thus if we wick-rotate $t \rightarrow i t,$ we get a 
one parameter family of (complex) solitons, namely, 
$S(r, \bar{r}, \theta) =  cos(\theta) X^s  + sin (\theta) Y^s$, where  $X^s(r,\bar{r}) = (x_1(r, \bar{r}) , it_1(r, \bar{r}), \phi_1(r, \bar{r}))$, $Y^s (r, \bar{r}) = (x_2(r, \bar{r}), it_2(r, \bar{r}), \phi_2(r, \bar{r})).$
We find the $F$ and $G$ functions corresponding to these complex solitons, 
(notation as in  ~\cite{W} page 617).

The process described here enables us to generate other solutions of the B-I, 
given one complex solution which  can be 
wick rotated to get a real minimal surface (which can be then be 
written in isothermal coordinates using 
the Weierstrass-Enneper repesentation). Then one can easily write the harmonic 
conjugate of the minimal surface in the same form and then make the 
one-parameter combination of the 
two mentioned above and wick rotate back to 
get the soliton family which starts from a soliton solution which is the 
initial solution with $ t \rightarrow - t$, (note that the B-I 
equation is invariant under $t \rightarrow -t$), and ends at a different 
soliton solution.  We give many examples of this process.

The paper is organized as follows.
We first give one example illustrating the case, namely that of the wick 
rotated helicoid and wick rotated the catenoid (since the catenoid is the  
harmonic conjugate of the helicoid).

Next we show that the first fundamental form, namely $E^s$, $G^s$ and $F^s$ 
are independent of $\theta$ and hence the action $A^s$ is invariant under 
$\theta$. This is due to a symmetry of the B-I equation which we explicitly 
show.

In the last section we give many examples illustrating the process described 
in the paper.

\section{The one parameter family of solitons}

Let $X (r, \bar{r})= (x_1(r, \bar{r}) , t_1(r, \bar{r}), \phi_1(r, \bar{r})) $ 

and $Y(r, \bar{r}) = (x_2(r, \bar{r}), t_2(r, \bar{r}), \phi_2(r, \bar{r})) $

be minimal surfaces which are harmonic conjugates of each other, given by the 
parameter $r$ and its conjugate. They are isothermal  in $r_1$ and $r_2$, where 
$r = r_1 + i r_2$.
Then we know that $cos(\theta) X  + sin (\theta) Y$ is a minimal surface for every $\theta,$ ~\cite{DoC}.

Then 
 $X^s (r, \bar{r})= (x_1(r, \bar{r}) , i t_1(r, \bar{r}), \phi_1(r, \bar{r})) $ 

 and  $Y^s(r, \bar{r}) = (x_2(r, \bar{r}), i t_2(r, \bar{r}), \phi_2(r, \bar{r})) $ are Born-Infeld  solitons for $imaginary$  time $i t_1$ and $i t_2$.

$X^s$ and $Y^s$ are complex solitons. The superscript 
$s$ stands for soliton.

\begin{proposition}
 $S_{\theta} = cos(\theta) X^s  + sin (\theta) Y^s$ are complex Born-Infeld 
solitons for every $\theta$.
\end{proposition}

\begin{proof} 
We will put $S_{\theta}$ in the form in ~\cite{W}, last section. 
According to ~\cite{D}

$X = (x_1(r, \bar{r}) ,  t_1(r, \bar{r}), \phi_1(r, \bar{r}))$ is a minimal 
surface implies

$x_1 - i t_1 = F_1(r) - \int \bar{r}^2 G_1^{\prime} ( \bar{r}) d \bar{r}$

$x_1 + it_1 = G_1(\bar{r}) - \int r^2 F_1^{\prime}(r) d r$

$ \phi_1 = \int r F_1^{\prime}(r) + \int \bar{r} G_1^{\prime}(\bar{r}) d \bar{r}$

where $F_1$ and $G_1$ are related by $F_1(r) = \overline{G_1(\bar{r})}$.

Similarly, 
$Y = (x_2(r, \bar{r}) ,  t_2(r, \bar{r}), \phi_2(r, \bar{r}))$ is a minimal 
surface implies,

$x_2 - i t_2 = F_2(r) - \int \bar{r}^2 G_2^{\prime} ( \bar{r}) d \bar{r}$

$x_2 + it_2 = G_2(\bar{r}) - \int r^2 F_2^{\prime}(r) d r$

$ \phi_2 = \int r F_2^{\prime}(r) + \int \bar{r} G_2^{\prime}(\bar{r}) d \bar{r}$

where $F_2$ and $G_2$ are related by $F_2(r) = \overline{G_2(\bar{r})}$.  

Then 
\begin{eqnarray*} 
& & S_{\theta} = (x_{\theta}^{s}, t_{\theta}^{s}, \phi_{\theta}^{s}) =cos(\theta) 
X^s + sin(\theta) Y^s   \\
& & = (cos(\theta) x_1 + sin(\theta) x_2 , i cos(\theta) t_1 + i sin(\theta) t_2, cos(\theta) \phi_1 + sin(\theta) \phi_2) 
\end{eqnarray*}
where recall superscript $s$ stands for soliton.
\begin{eqnarray*} 
x_{\theta}^{s} - t_{\theta}^{s} &=& cos(\theta) F_1(r) + \sin(\theta) F_2(r)  \\
&-&  \int (\bar{r}^2 (cos(\theta) G_1^{\prime} (\bar{r}) + \sin(\theta) G_2^{\prime} ( \bar{r})) d \bar{r} \\
&=& F_{\theta}^{s} (r) -  \int \bar{r}^2 G_{\theta} ^{s \prime}(\bar{r}) d \bar{r}
\end{eqnarray*}
where $ F_{\theta}^{s} (r) = cos(\theta) F_1(r) + \sin(\theta) F_2(r)$ and 
 $G_{\theta}^{s}(\bar{r})= cos(\theta) G_1 (\bar{r}) + \sin(\theta) G_2 ( \bar{r})$. 

One can easily check that

$x_{\theta}^{s} + t_{\theta}^{s} =  G_{\theta}^{s} (\bar{r}) -  \int r^2 F_{\theta}^{s \prime}(r) d r $

$\phi_{\theta}^s = \int r F_{\theta}^{s \prime} (r)  + \int \bar{r} G_{\theta}^{s \prime} (\bar{r}).$

Renaming variables, $\bar{r} = s$, we get this is exactly in the form of 
solutions to the Born Infeld equation as in ~\cite{W}, page 617.
Thus $S_{\theta}$ is a (complex) Born-Infeld soliton.
\end{proof}

\begin{corollary}
The partial derivatives of $S_{\theta}$ with respective to $\theta$ are also soliton 
solutions.
\end{corollary}

\begin{proof}

$\frac{\partial S_{\theta}}{\partial \theta} = cos(\theta + \frac{\pi}{2})  X^s + sin(\theta + \frac{\pi}{2}) Y^s$

$ \frac{\partial^2 S_{\theta}}{\partial^2 \theta} = cos(\theta + \pi)  X^s + sin(\theta + \pi) Y^s$

$ \frac{\partial^3 S_{\theta}}{\partial^3 \theta} = cos(\theta -\frac{\pi}{2}) X^s + sin(\theta - \frac{\pi}{2}) Y^s$

 $      \frac{\partial^4 S_{\theta}}{\partial^4 \theta}     =S_{\theta}$ 

These are again of the form  $cos(\theta_0)  X^s + sin(\theta_0) Y^s$ and thus are soliton solutions.
\end{proof}

\section{An example:}

Let us write the catenoid and the helicoid (two conjugate minimal surfaces)
in a variant of their  Weirstrass-Enneper representation, ~\cite{D}, ~\cite{N},  which is also isothermal.

\begin{proposition}
a) The helicoid can be written in a parametrised form in the following way:

$x_1 = -\frac{1}{2} \rm{Im} ( r + \frac{1}{r})$

$t_1 = \frac{1}{2} \rm{Re} ( r - \frac{1}{r})$

$\phi_1 =  \rm{Im} (ln r) $

b) The catenoid can be written in a parametrised form in the following way:

$x_2 = \frac{1}{2} \rm{Re} ( r + \frac{1}{r} )$

$t_2 = \frac{1}{2} \rm{Im} ( r - \frac{1}{r}) $

$ \phi_2 = - \rm{Re} (ln r)$
\end{proposition}

\begin{proof}
a) The non parametric form of helicoid is $\phi(x,t)= \tan^{-1}\frac{t}{x}$. 
 As $\phi_x=\frac{-t}{x^2+t^2}$ and $\phi_t= \frac{x}{x^2+t^2}$, we have $u=\phi_{\bar{z}}= \phi_x x_{\bar{z}}+ \phi_t t_{\bar{z}}.$

That is $u =\frac{-t+ix}{2(x^2+t^2)}= \frac{iz}{2|z|^2}$, where $z = x + i t$. 

\begin{equation}u=\frac{i}{2\bar{z}}\label{u}\end{equation}

Similarly we have 

\begin{equation}v=\frac{-i}{2z}\label{v}\end{equation}

Let us make the following coordinate change, ~\cite{D}, ~\cite{W}:
\begin{equation} 
 r = \frac{\sqrt{1 + 4 u v } - 1}{2 v}. \label{r}
\end{equation}
Then 

\begin{equation} u=\frac{r}{1-|r|^2}\text{ and } v=\frac{\bar{r}}{1-|r|^2}\label{uv}\end{equation}

Equation \ref{u}, \ref{v} and \ref{uv} gives

\begin{equation} z=\frac{i}{2}(r-\frac{1}{\bar{r}})\label{z}\end{equation}

which in turn gives

\begin{equation}
x= -\frac{1}{2} \rm{Im} (r+\frac{1}{r})\text{ and }t=\frac{1}{2} \rm{Re}(r-\frac{1}{r})\end{equation}

also from equation \ref{z}, we have $F(r)= \frac{i}{2r}$ and hence $G(\bar{r})=\frac{-i}{2\bar{r}}$, ~\cite{D}. 

Then we have $\phi(r)= \int r F'(r) dr+\int \bar{r} G'(\bar{r})d\bar{r}$, ~\cite{D}, and thus
$\phi(r)= \frac{-i}{2}[\ln r - \ln\bar{r}]$, that is we have:

\begin{equation}
\phi(r)=\rm{Im}(\ln r)\label{phi}
\end{equation}

(b) The nonparametric form of catenoid is

$\phi(x,t)= \cosh^{-1}\sqrt{x^2+t^2}$.
As seen in helicoid case, for the catenoid we have:\\

 As $\phi_x=\frac{x}{\sqrt{x^2+t^2-1}\sqrt{x^2+t^2}}$ and $\phi_t= \frac{t}{\sqrt{x^2+t^2-1}\sqrt{x^2+t^2}}$, and $$u=\phi_{\bar{z}}= \phi_x x_{\bar{z}}+ \phi_t t_{\bar{z}}=\frac{\phi_x+i\phi_t}{2}$$

That is $u =\frac{z}{2\sqrt{x^2+t^2-1}\sqrt{x^2+t^2}}.$

Again with the same coordinate change as in equation \ref{r},  \ref{uv}  and $u$  as above we have $\frac{z}{\bar{z}}=\frac{r}{\bar{r}}$, that is: \begin{equation}z=\frac{r}{\bar{r}}\bar{z}.\label{z2}\end{equation}

Now as we have 
$$u =\frac{z}{2\sqrt{x^2+t^2-1}\sqrt{x^2+t^2}}=\frac{z}{2\sqrt{|z|^2-1}\sqrt{ |z|^2}}.$$ That is $$\frac{r}{1-|r|^2}= \frac{z}{2\sqrt{|z|^2-1}\sqrt{|z|^2}}$$ Squaring it we have

$$\frac{z^2}{4(|z|^2-1).|z|^2}=\frac{r^2}{(1-|r|^2)^2}$$

Using equation \ref{z2}, we have $$4 |r|^2(\frac{r}{\bar{r}} \bar{z}^2-1)=(1-|r|^2)^2$$
That is $$\bar{z}^2= \frac{\bar{r}}{r}\left(\frac{(1-|r|^2)^2}{4|r|^2}+1\right)$$

$$\bar{z}= \pm\frac{1}{2}(\bar{r}+\frac{1}{r}).$$ We take the positive sign, 
because this gives us the non-parametric form. Hence in this case we have:

\begin{equation*} x= \frac{1}{2} \rm{Re}(r+\frac{1}{r})\text{, } \\
t=\frac{1}{2} \rm{Im}(r-\frac{1}{r})\text{, }\\
\phi(r)= -\rm{Re}(\ln r)\label{catenoid}
\end{equation*}
\end{proof}

It is easy to check that the catenoid is conjugate harmonic to the helicoid 
 because

$x_1 + i x_2 = i ( r + \frac{1}{r}) $ 

$ t_1 + i t_2 = r - \frac{1}{r} $

$ \phi_1 + i \phi_2 = - i ln r $

so that the right hand sides of all the expressions are analytic functions of 
the complex variable $r$.

\begin{proposition} 
$F^s_{\theta} = \frac{i}{2} \frac{e^{-i \theta}}{r}$ and 
$G^s_{\theta} = \frac{-i}{2} \frac{e^{i \theta}}{\bar{r}}$
are the $F$ and $G$ functions for our family of soliton solutions.
\end{proposition}

\begin{proof}
$x_{\theta}^s = cos(\theta) x_1 + sin(\theta) x_2 ,$
$t_{\theta}^s = i  cos(\theta) t_1 + i sin(\theta) t_2,$
$\phi_{\theta}^s = cos(\theta) \phi_1 + sin(\theta) \phi_2.$

$x_{\theta}^s - t_{\theta}^s = cos(\theta) (x_1 - i t_1)  + sin(\theta) (x_2 - i t_2)$

$x_{\theta}^s + t_{\theta}^s = cos(\theta) (x_1 + i t_1)  + sin(\theta) (x_2 + i t_2)$

$x_1 - i t_1 = -\frac{i}{2} (\bar{r} - \frac{1}{r})$

$x_2 - i t_2 = \frac{1}{2} ( \bar{r} + \frac{1}{r})$

$x_{\theta}^s - t_{\theta}^s = -\frac{i}{2} \bar{r} e^{i \theta} + \frac{i}{2} \frac{e^{-i \theta}}{r}$. 

$x_{\theta}^{s} + t_{\theta}^{s} = \frac{i}{2} r e^{-i \theta} - \frac{i}{2} \frac{e^{i \theta}}{\bar{r}}$

Thus $F_{\theta}^s(r) = \frac{i}{2} \frac{e^{-i \theta}}{r}$ and $G^s(\bar{r})  = - \frac{i}{2} \frac{e^{i \theta}}{\bar{r}}.$

We can check that $F_{\theta}^s(r) = \overline{G_{\theta}^{s} ( \bar{r})}.$  

Recall:

$\phi_{\theta}^{s} = \int r F_{\theta}^{s \prime} (r)d r + \int \bar{r} G_{\theta}^{s \prime} (\bar{r}) d \bar{r}.$

Thus

$\phi_{\theta}^{s} =  - \frac{i}{2} (ln r)  e^{- i \theta} + \frac{i}{2} (ln \bar{r}) e^{i \theta}. $

If $\theta = 0$ this corresponds to the wick rotated helicoid, namely
$\phi_0^s =   \rm{Im} (ln r) $

and if $\theta = \frac{\pi}{2}$, this corresponds to the  wick rotated 
catenoid, namely,

$\phi_{\frac{\pi}{2}}^s = - \rm{Re} (ln r) $
\end{proof}

\section{$\theta$-invariants}

Let $X^s_{\theta}= (x^s_{\theta},  t^s_{\theta}, \phi^s _{\theta})$ be a soliton 
solution as before. 

We show that the coefficients of the first fundamental form, and hence the 
Born-Infeld action is independent of $\theta$.

\begin{proposition}
Let $r = r_1 + i r_2$. Then 
$E^s= x^{s2}_{\theta, r_1} - t^{s 2}_{\theta, r_1} + \phi^{s2} _{\theta, r_1} $ remains invariant with respect to $\theta$. 
Similarly, 
$G^s= x^{s 2}_{\theta, r_2} - t^{s 2}_{\theta, r_2} + \phi^{s2} _{\theta, r_2} $  remains invariant with respect to $\theta$. 
Also, $ F^s = x^s_{\theta, r_1} x^s_{\theta, r_2} - t^s_{\theta, r_1} t^s_{\theta, r_2} + \phi^s _{\theta, r_1} \phi^s _{\theta, r_2} =0$ for all $\theta$.
Thus $A^s = \int \sqrt{E^sG^s - F^{s2}} d r_1 d r_2 = \int \sqrt{1 + \phi^{s2}_{x^s} - \phi^{s2}_{t^s}} dx^s dt^s$ is $\theta$ invariant.
\end{proposition}

\begin{proof}
We have $$X^s_{\theta}= X_1^s\cos\theta +X_2^s\sin\theta$$ where corresponding $X_1$ and  $X_2$ are harmonic conjugate minimal surafces 
in $r_1$ and $r_2$ variable, and
$$\frac{\partial X_1}{\partial r_1}= \frac{\partial X_2}{\partial r_2}\text{ and } \frac{\partial X_1}{\partial r_2}= -\frac{\partial X_2}{\partial r_1}.$$
If $X_i(r_1,r_2)= (x_i, t_i, \phi_i)$, we have
 $$X^s_{\theta}=( x_1(r,s)\cos\theta+x_2(r,s)\sin\theta, i(t_1(r,s)\cos\theta$$
 $$+t_2(r,s)\sin\theta), \phi_1(r,s)\cos\theta+\phi_2(r,s)\sin\theta)$$

As $X_1$ and $X_2$ are conjugate we have:

$$\frac{\partial X_1}{\partial r_1}= \frac{\partial X_2}{\partial r_2}\text{ and } \frac{\partial X_1}{\partial r_2}= -\frac{\partial X_2}{\partial r_1}.$$
Then
\begin{eqnarray*}
 x^{s2}_{\theta, r_1} - t^{s 2}_{\theta, r_1} + \phi^{s2} _{\theta, r_1} 
 &=&x^{2}_{\theta, r_1} + t^{ 2}_{\theta, r_1} + \phi^{2} _{\theta, r_1}  \\
&=& (X_{1r_1}\cos\theta+X_{2r_1}\sin\theta).(X_{1r_1}\cos\theta+X_{2r_1}\sin\theta) \\
&=& (X_{1r_1}\cos\theta-X_{1r_2}\sin\theta).(X_{1r_1}\cos\theta-X_{1r_2}\sin\theta)\\
&=& X_{1r_1}.X_{1r_1}\cos^2\theta+\sin^2\theta X_{1r_2}.X_{1r_2}\\
& & +\cos\theta\sin\theta X_{1r_1}.X_{1r_2}-\cos\theta\sin\theta X_{1r_1}.X_{1r_2} 
\end{eqnarray*}

Now we have $ X_{1r_1}.X_{1r_1}= X_{1r_2}.X_{1r_2}$, (since $r_1$ and $r_2$ are 
isothermal coordinates for $X_1$), 

\begin{eqnarray*}
E^s &=& x^{s2}_{\theta, r_1} - t^{s 2}_{\theta, r_1} + \phi^{s2} _{\theta, r_1} \\
&=& X_{1r_1}.X_{1r_1}
\end{eqnarray*}

  Hence $E^s$ is independent of $\theta$.

\begin{eqnarray*}
& & x^s_{\theta, r_1} x^s_{\theta, r_2} - t^s_{\theta, r_1} t^s_{\theta, r_2} + \phi^s _{\theta, r_1} \phi^s _{\theta, r_2} \\
&=&x_{\theta, r_1} x_{\theta, r_2}+ t_{\theta, r_1} t_{\theta, r_2} + \phi_{\theta, r_1} \phi_{\theta, r_2}\\
&=&(X_{1r_1}\cos\theta+X_{2r_1}\sin\theta).(X_{1r_2}\cos\theta+X_{2r_2}\sin\theta)\\
&=&(X_{1r_1}\cos\theta-X_{1r_2}\sin\theta).(X_{1r_2}\cos\theta+X_{1r_2}\sin\theta)\\
&=& X_{1r_1}.X_{1r_2}\cos^2\theta-\sin^2\theta X_{1r_2}.X_{1r_1}\\
& & +\cos\theta\sin\theta X_{1r_1}.X_{1r_1}-\cos\theta\sin\theta X_{1r_2}.X_{1r_2}
\end{eqnarray*}
Again $ X_{1r_1}.X_{1r_1}= X_{1r_2}.X_{1r_2}$ and $X_{1r_1}.X_{1r_2}=0$, we have $F^s=0$.

Similiary we can prove for $G^s$.    Hence we see that $E^s, F^s, G^s$ all are independent of $\theta$ which in turn gives $A^s$ is independent of $\theta$.

\end{proof}

{\bf Lorentz Invariance of the Born-Infeld equation}

There is a well-known symmetry, namely, the Lorentz invariance of the 
Born-Infeld equation which is reponsible for these invariant quantities.
We rederive it here.

\begin{proposition}
There is a symmetry in the Born-Infeld equation, namely
if 
$\left[\begin{array}{c}
x^{\prime}  \\
t^{\prime} 
\end{array}
\right] = \left[\begin{array}{cc}
                        cosh(\theta) & sinh(\theta)  \\
                         sinh(\theta) & cosh(\theta)             
                       \end{array} \right] \left[\begin{array}{c}
                                                  x  \\
                                                  t 
                                                \end{array} \right] $
then $\phi (x^{\prime}, t^{\prime}) $ satisfies the same B-I equation with
$x$ and $t$ replaced by $x^{\prime}$ and $ t^{\prime}$. 
\end{proposition}

\begin{proof}
Let $\left(
       \begin{array}{cc}
         a & b \\
         c & d \\
       \end{array}
     \right),$ $ad-bc \neq 0,$ denote the symmetry to Born-Infield equation, then we have: $\phi_{x'}= a\phi_x+c\phi_t,$   $\phi_{x'x'}= a^2\phi_{xx}+c^2\phi_{tt}+2ac\phi_{xt},$ $\phi_{t'}= b\phi_x+d\phi_t$, $\phi_{t't'}= b^2\phi_{xx}+d^2\phi_{tt}+2bd\phi_{xt}$, and
$\phi_{x't'}= ab\phi_{xx}+cd\phi_{tt}+(bc+ad)\phi_{xt}$.  Hence B-I equation for $\phi(x',t')$ changes to
\begin{eqnarray}& &(1-\phi_{t'}^2)\phi_{x'x'}+2\phi_{x'}\phi_{t'}\phi_{x't'}-(1+\phi_{x'}^2)\phi_{t't'}\label{B-I'}\\
&=&[1-(b\phi_x+d\phi_t)^2](a^2\phi_{xx}+c^2\phi_{tt}+2ac\phi_{xt})+2(a\phi_x+c\phi_t)(b\phi_x+d\phi_t)\nonumber\\
& &[ab\phi_{xx}cd\phi_{tt}+(ad+bc)\phi_{xt}]-[1+(a\phi_x+c\phi_t)^2](b^2\phi_{xx}+d^2\phi_{tt}+2bd\phi_{xt})\nonumber
\end{eqnarray}
In above expression \ref{B-I'}, coefficient of $\phi_{xx}$ is
\begin{eqnarray*}
& & a^2-(b\phi_x+d\phi_t)^2a^2+2ab(a\phi_x+c\phi_t)(b\phi_x+d\phi_t)-b^2-(a\phi_x+c\phi_t)^2b^2\\
&=&(a^2-b^2)+\phi_x^2(a^2b^2+2a^2b^2-a^2b^2)+\phi_t^2(-a^2d^2+2abcd-b^2c^2)\\
& & +\phi_x\phi_t(-bda^2+2abad+2abcb-2abb^2)\\
&=&(a^2-b^2)-\phi_t^2(a^2d^2+b^2c^2-2abcd)
\end{eqnarray*}
Hence for invariance of B-I equation we must have $a^2-b^2=1$, and $a=d$, $b=c$.  With these condition coefficient for $\phi_{xx}$ in equation \ref{B-I'} will be $$(a^2-b^2)-\phi_t^2(a^2d^2+b^2c^2-2abcd)= (a^2-b^2)[1-\phi_t^2(a^2-b^2)]= 1-\phi_t^2$$

When $a^2-b^2=1$ $a=d$ and $b=c$, we have coefficient of $\phi_{xt}$ in equation \ref{B-I'} as $ 2 \phi_{xt}$.

In the same way the coefficient of $\phi_{tt}$ in equation \ref{B-I'} $=(1+\phi_x)^2$.   Hence \ref{B-I'} changes to B-I equation in $\phi(x,t)$ that is we have $\phi(x',t')$  is a soliton if and only if $\phi(x,t)$ is a solition. Also we have $a^2-b^2= 1 $ if and only if $a=\cosh\theta$ and $b=\sinh\theta$.

That is under the coordinate change $\left[\begin{array}{cc}
                        cosh(\theta) & sinh(\theta)  \\
                         sinh(\theta) & cosh(\theta)
                       \end{array} \right]$,  solution to the Born-Infield equation remain invariant.
\end{proof}

It is easy to check that this symmetry keeps $A^s$ invariant.
This is expected  
 since the B-I equation is obtained by minimizing this 
action.

\section{Many more examples}

Recall the Weierstrass-Enneper representation of minimal surfaces, 
namely, in the neighborhood of a
nonumbilic interior point, any minimal surface can
be represented in terms of $w$ as follows, ~\cite{N},
\begin{eqnarray*}
x(\zeta) &=& x_0 + \rm{Re} \int_{\zeta_0}^{\zeta} (1-w^2) R(w) \,d w\\
t(\zeta) &=& t_0 + \rm{Re} \int_{\zeta_0}^{\zeta} i (1 + w^2) R(w) \,d w\\
\phi(\zeta) &=& \phi_0 + \rm{Re} \int_{\zeta_0}^{\zeta} 2 w R(w) \,d w
\end{eqnarray*}

This is an isothermal representation (w.r.t.  $\zeta_1$ and $\zeta_2$ where 
$\zeta = \zeta_1 + i \zeta_2.$)

Various examples of minimal surfaces are as follows, ~\cite{N}, page 148.  

$R(w) = 1$ leads to the Enneper minimal surface.

$R(w) = \frac{\kappa}{2 w^2}$, $\kappa$ real, leads to the catenoid, 
$\frac{z}{\kappa} = cosh^{-1} ( \frac{\sqrt{x^2 + t^2}}{|\kappa|})$. 

$R(w) = \frac{i \kappa}{2 w^2}$, $\kappa$ real, leads to the right helicoid
$\frac{z}{\kappa} = tan^{-1} ( \frac{x}{t})$. 

$R(w) = \frac{\kappa e^{i\alpha}}{2 w^2}$ leads to the general helicoid.

$R(w) = \frac{2}{(1 - w^4)}$ leads to the Scherk's minimal surface.

$R(w) = \frac{-2aisin(2 \alpha)}{ (1 + 2 w^2 cos(2 \alpha) + w^4)} $ , 
$0 < \alpha < \pi/2$, $a >0$ leads to the general Scherk's minimal surface.

$R(w) = 1 - w^{-4}$ (and substituting $-t$ for $t$) leads to the Henneberg surface.

$R(w) = \frac{i a (w^2 -1)}{w^3} - \frac{ib}{2 w^2}$, $a$ and $b$ real, and setting $w = e^{-i \gamma/2} $, leads to the general Enneper surface and , in particular, for $a=1$ and $b=0$, to the Catalan's surface.

$R(w) = (1 - 14 w^4 + w^8)^{-1/2}$ leads to the Schwarz-Riemann minimal surface.

Description and pictures of these minimal surfaces can be found in  ~\cite{N}.

These are in isothermal representation. 

To find their $harmonic$ $conjugate$ minimal surfaces, we need to replace $R(w) $ by $-i R(w)$. 

Because if 
\begin{eqnarray*}
x_1(\zeta) &=& x_{01} + \rm{Re} \int_{\zeta_0}^{\zeta} (1-w^2) R(w) \,d w\\
t_1(\zeta) &=& t_{01} + \rm{Re} \int_{\zeta_0}^{\zeta} i (1 + w^2) R(w) \,d w\\
\phi_1(\zeta) &=& \phi_{01} + \rm{Re} \int_{\zeta_0}^{\zeta} 2 w R(w) \,d w
\end{eqnarray*}

and 
\begin{eqnarray*}
x_2(\zeta) &=& x_{02} + \rm{Re} (-i \int_{\zeta_0}^{\zeta} (1-w^2) R(w) \,d w)\\
           &=& x_{02} + \rm{Im}  \int_{\zeta_0}^{\zeta} (1-w^2) R(w) \,d w\\
t_2(\zeta) &=& t_{02} + \rm{Re} (-i \int_{\zeta_0}^{\zeta} i (1 + w^2) R(w) \,d w)\\
&=& t_{02} + \rm{Im}  \int_{\zeta_0}^{\zeta} i (1 + w^2) R(w) \,d w\\
\phi_2(\zeta) &=& \phi_{02} + \rm{Re} (-i \int_{\zeta_0}^{\zeta} 2 w R(w) \,d w)\\
 &=& \phi_{02} + \rm{Im}  \int_{\zeta_0}^{\zeta} 2 w R(w) \,d w
\end{eqnarray*}

then,

$x_1 + i x_2 = x_{01} + i x_{02} + \int_{\zeta_0}^{\zeta} (1-w^2) R(w) \,d w $

$t_1 + i t_2 =  t_{01} + i t_{02} +  \int_{\zeta_0}^{\zeta} i (1 + w^2) R(w) \,d w$

$\phi_1 + i \phi_2 =  \phi_{01} + i  \phi_{02} +  \int_{\zeta_0}^{\zeta} 2 w R(w) \,d w .$

Since the right-hand side are holomorphic functions of 
$\zeta = \zeta_1 + i \zeta_2$, 

$(x_2, t_2, \phi_2)$ is harmonic conjugate of  $(x_2, t_2, \phi_2)$ are 
and the representations above are isothermal (w.r.t. $\zeta_1$ and $\zeta_2$) .

Thus we can combine $cos {\theta} (x_1, t_1, \phi_1) + sin{\theta} (x_2, t_2, \phi_2)$ and get another minimal surface.

By ``wick rotating'', namely, $ t \rightarrow i t$, we get a one -paramter family of solitons, 
$cos {\theta} (x_1, it_1, \phi_1) + sin{\theta} (x_2, it_2, \phi_2)$

Each choice of $R(w)$ gives us an example. Thus we get many examples.

{\bf Remark:} We re-emphasize that 
the process described here enables us to generate other solutions of the B-I, 
given one complex solution which  can be
wick rotated to get a real minimal surface (which can be then be 
written in isothermal coordinates using 
the Weierstrass-Enneper repesentation). Then one can easily write the harmonic 
conjugate of the minimal surface in the same form and then make the 
one-parameter combination of the 
two mentioned above and wick rotate back to 
get the soliton family which starts from a soliton solution which is the initial solution with $ t \rightarrow - t$, (note that the B-I 
equation is invariant under $t \rightarrow -t$), and ends at a different 
soliton solution.  We have given many examples of this process. 

{\bf Remark:} We are using the word solition for solutions of the B-I equations.
But since these are complex solutions, they need not be actual solitons.

{\bf Remark:} Given a minimal surface in isothermal coordinates, its harmonic conjugate in isothermal coordinates is also a minimal surface.
This is because  ${\bf{X} = \bf{X}} (u,v)$ is a minimal surface iff
${\bf X}$ is isothermal (w.r.t $u$ and $v$) and harmonic, ~\cite{O}. 
(Here ${\bf{X}} (u,v) = (x(u,v), t(u,v), \phi(u,v))$.)

\vspace{.5in}

{\bf Correction:}  There are  corrections in ~\cite{D}. 
Equation $(14)$ should read 
$ \bar{z} = \bar{z}_0 + F(\bar{\zeta}) - \int \bar{\zeta}^2 G^{\prime}(\bar{\zeta}).$

Here $F(r) = \overline{G(\bar{r})} .$

Also, in ~\cite{D} our representation is a little different from the 
Weierstrass-Enneper representation, though both are isothermal. The domain of 
validity 
of  the W-E representation is away from the umbilical points, namely, 
$\phi_{x x} \phi_{y y} - \phi_{x y }^2 =0,$
while our 
representation fails where $\phi_{z z} \phi_{\bar{z} \bar{z}} - \phi_{\bar{z} z}^2 =0.$

\vspace{.5in}

{\bf Acknowledgement:} 
The first author  would like to thank Professor Randall Kamien for the 
observation that
the minimal surface equation is just the wick rotated Born-Infeld equation.


\begin{thebibliography}{99}

\bibitem{D}R.Dey: The Weierstrass-Enneper representation using hodographic 
coordinates on a minimal surfaces;
 Proc. of Indian 
Acad. of Sci. -- Math.Sci. Vol.113, No.2, May (2003), 
pg 189-193; math.DG/0309340.
\bibitem{DoC} Do Carmo M.: Differential Geometry of Curves and Surfaces,
Prentice Hall, 1976.
\bibitem{N} Nitsche J.C.C.: Lectures on Minimal surfaces, Volume
1, Cambridge University Press, 1989, Cambridge.
\bibitem{O} Osserman R. : Survey of Minimal Surfaces; Dover Publications,
New York, 1986, New York.
\bibitem{W} Whitham G.B.: Linear and Nonlinear Waves; John Wiley and Sons, 1999, New York.
\end{thebibliography}
\end{document}